\documentclass{article}[11]
\usepackage{amsmath, amssymb, color}

\newtheorem{thm}{Theorem}[section]
\newtheorem{prop}[thm]{Proposition}
\newtheorem{lem}[thm]{Lemma}
\newtheorem{cor}[thm]{Corollary}

\newenvironment{remark}{\medskip\refstepcounter{thm}
\noindent{\bf Remark \thesection.\arabic{thm}\ }}{\medskip}

\newenvironment{proof}[1][,]{\medskip\ifcat,#1
\noindent{{\it Proof}:\ }\else\noindent{\it Proof of #1.\ }\fi}
{\hfill$\square$\medskip}

\DeclareMathOperator\vol{vol}

\DeclareMathOperator\GL{GL}

\DeclareMathOperator\Div{div}
\DeclareMathOperator\Diff{Diff}
\DeclareMathOperator\Ric{Ric}

\DeclareMathOperator\Vol{Vol}

\DeclareMathOperator\tnabla{\widetilde{\nabla}}
\def\d{{\rm d}}
\def\w{\wedge}
\DeclareMathOperator\tr{tr}
\def\C{\mathbb{C}}
\def\tU{\widetilde{\mathcal{U}}}
\def\tL{\widetilde{\mathcal{L}}}
\def\tF{\tilde{F}}
\def\CP{\mathbb{CP}}
\def\Sph{\mathbb{S}}
\def\R{\mathbb{R}}

\def\hZ{\hat{Z}}
\def\hW{\hat{W}}
\def\hiota{\hat{\iota}}
\DeclareMathOperator\Graph{Graph}

\begin{document}

\title{From minimal Lagrangian to $J$-minimal submanifolds: persistence and uniqueness}

\author{Jason D. Lotay\footnote{University College London, U.K.,  \texttt{j.lotay@ucl.ac.uk}}  $\;$and Tommaso Pacini\footnote{University of Torino, Italy,
\texttt{tommaso.pacini@unito.it}}}


\maketitle

\abstract{Given a minimal Lagrangian submanifold $L$ in a negative K\"ahler--Einstein manifold $M$, we show that any small K\"ahler--Einstein perturbation of $M$ induces a deformation of $L$ which is minimal Lagrangian with respect to the new structure. This provides a new source of examples of minimal Lagrangians. More generally, the same is true for the larger class of totally real $J$-minimal submanifolds in K\"ahler manifolds with negative definite Ricci curvature.}

\section{Introduction}\label{s:intro}

Let $(M,\overline{g})$ be a Riemannian manifold. Recall that a submanifold $\iota:L\rightarrow M$ is \textit{minimal} if it is a critical point of the Riemannian volume functional, \textit{i.e.}~if the mean curvature vector field $H$ vanishes. 

Now let $(M^{2n},\overline{\omega})$ be a $2n$-dimensional symplectic manifold. One says that a $n$-dimensional submanifold $\iota:L^n\rightarrow M$ is \textit{Lagrangian} if the pull-back $\iota^*\overline{\omega}$ of the symplectic form vanishes. 

It should be expected that, whenever $\overline{g}$, $\overline{\omega}$ are somehow compatible, these two conditions will interact   so that submanifolds which are both minimal and Lagrangian will have significant geometric properties.

The standard setting for this to happen is that of K\"ahler manifolds $(M,J,\overline{g},\overline{\omega})$. Here, minimal Lagrangian submanifolds are a classical topic and several results are known. 
However, it was first pointed out by Bryant \cite{Bryant} that imposing both these conditions generally results in an over-constrained problem. Geometrically this corresponds to the calculation that the ambient Ricci 2-form 
$\bar\rho$ must necessarily vanish along such submanifolds, \textit{i.e.}~$\iota^*\bar\rho=0$. The best way to remove this additional constraint on the submanifold is to assume $M$ is K\"ahler--Einstein (KE), \textit{i.e.}~$\bar\rho=\lambda\bar\omega$ 
for some constant $\lambda$: if $\lambda=0$ this eliminates the Ricci curvature globally, if $\lambda\neq 0$ it makes the Lagrangian and Ricci vanishing conditions coincide.

Assuming the KE condition on $M$, the main classical facts known about minimal Lagrangians are the following.

\begin{itemize}
 \item The Lagrangian condition is preserved under mean curvature flow. 
 \item In KE manifolds which are Ricci-flat (more precisely, Calabi--Yau manifolds), minimal Lagrangians are calibrated. Hence, they have the stronger property of minimizing volume in their homology class. 
In this setting there is a good deformation theory for minimal Lagrangians.
 \item In KE manifolds with negative Ricci curvature, minimal Lagrangians are strictly stable: the second variation of the Riemannian volume functional is (uniformly) strictly 
 positive at a minimal Lagrangian. 
\end{itemize}
In \cite{LPcoupled} and \cite{LPtotally} we initiated a new point of view on minimal Lagrangians, inserting them into the broader context of totally real submanifolds. In particular, building  
upon previous work in \cite{Borrelli} and \cite{Smoczyk}, we achieved the following.
\begin{itemize}
 \item We related the standard Riemannian volume functional and the corresponding mean curvature vector field $H$ to a modified $J$-volume functional and to a corresponding $J$-mean curvature vector field $H_J$: 
these quantities coincide on Lagrangians, but the modified versions have better properties on the larger class of totally real submanifolds, cf.~\cite{LPcoupled}.
 \item We defined new geometric flows extending the study of Lagrangian mean curvature flow in KE manifolds to more general ambient manifolds and to totally real submanifolds, cf.~\cite{LPcoupled}.
 \item In K\"ahler manifolds with non-positive Ricci curvature, we proved that the $J$-volume functional is convex with respect to a certain notion of geodesics in the space of totally real submanifolds, cf.~\cite{LPtotally}.
 \item We outlined an analogy between minimal Lagrangian geometry in negative KE manifolds and the geometry of K\"ahler metrics with constant scalar curvature, in terms of complexified diffeomorphism groups, convex functionals 
and moment maps, cf.~\cite{LPtotally}.
\end{itemize}
The paper at hand has two main goals. 

First, we wish to further develop the analytic theory of minimal Lagrangian submanifolds by showing that, in the negative KE case, any such submanifold will continue to exist under small KE perturbations of the ambient structure. Up to now this question was largely open, being known to hold only in the special case of negative KE surfaces \cite{YILee}. The techniques used there however are specific to that dimension and cannot be generalised. Our Theorem \ref{thm:persistence} proves persistence in all dimensions.

Second, we wish to show that, when $M$ has definite Ricci curvature (not necessarily KE), totally real geometry offers a natural replacement for the over-constrained minimal Lagrangian condition: 
\textit{$J$-minimal} submanifolds, \textit{i.e.}~the critical points of the $J$-volume functional. The two classes of submanifolds (minimal Lagrangian and $J$-minimal submanifolds)
 coincide in the KE case. 
Geometric manifestations of this point of view already appear in \cite{LPcoupled} and \cite{LPtotally}, but developments on the analytic side were blocked by the fact that the $J$-minimal equation is not elliptic. 
 The starting point of our analytic study here is the simple observation, cf.~Theorem \ref{thm:minlag_is_Jmin}, that such submanifolds are automatically Lagrangian with respect to $\bar\rho$, 
viewed as an alternative symplectic form on $M$ (but are not necessarily Lagrangian with respect to $\bar\omega$). As a consequence we show that, in the negative definite case, 
$J$-minimal submanifolds share the same uniqueness and persistence properties as minimal Lagrangians, cf.~Theorems \ref{thm:uniqueness} and \ref{thm:persistence_bis}. This leads to possibly the first examples of 
compact $J$-minimal submanifolds in non-KE ambient spaces.

The paper is organized as follows. Sections \ref{s:cpxvol} and \ref{s:grassmannian} review the geometry of totally real submanifolds following \cite{LPcoupled}, but also offer a new perspective on that theory and new results,
 cf.~Proposition \ref{prop:diff_Omega}, Remark \ref{rem:calibration} and (part of) Theorem \ref{thm:minlag_is_Jmin}. Section \ref{s:linearisation} is the technical core of the paper. 
Our main results are presented in Section \ref{s:Lag.KE} and applied to examples in Section \ref{s:examples}.

\paragraph{Acknowledgements.} We would like to thank Claude LeBrun for suggesting the problem of persistence of minimal Lagrangians to us, Andr\'e Neves for informing us of the reference \cite{White}, and 
Simon Donaldson, Nicos Kapouleas and Cristiano Spotti for interesting discussions.

JDL was partially supported by EPSRC grant EP/K010980/1. TP thanks the Scuola Normale Superiore, in Pisa, for hospitality and research funds.

This paper is dedicated to Paolo de Bartolomeis, who was also TP's advisor. His excellent lectures conveyed the idea that Geometry is not just a body of results. It is also a point of view, which can provide a guiding light in many other fields of Mathematics and Science.  

\section{Complex volumes and submanifolds}\label{s:cpxvol}
In order to achieve a better perspective on Lagrangian and minimal Lagrangian submanifolds, in \cite{LPcoupled} and \cite{LPtotally} we developed the differential geometry of the much larger class of totally real submanifolds. 
The goal of this and of the next section is to summarize and extend those results, also adopting a slightly different point of view, in terms of Grassmannians: this allows for a more flexible set of formulae concerning the 
 derivatives of complex volume forms.

\paragraph{Linear algebra.} Let $(V,J,h=\bar{g}-i\bar{\omega})$ be a Hermitian vector space of complex dimension $n$.

Recall that a (normalized) \textit{complex volume} form on $V$ is an $(n,0)$-form $\Omega$ such that $|\Omega|_h=1$.
A subspace $\pi\leq V$ of real dimension $n$ is called \textit{totally real} if $\pi\cap J\pi=\{0\}$, \textit{i.e.}~$V\simeq \pi\otimes\C$. This yields a splitting $V=\pi\oplus J\pi$. We will denote the corresponding projection maps by
\begin{equation*}
 \pi_L:V\rightarrow V\ (\mbox{with image }\pi),\ \ \ \pi_J:V\rightarrow V\ (\mbox{with image }J\pi).
\end{equation*}
The reason for this notation will become apparent below, when we introduce submanifolds $L$ and focus on the case $\pi=T_pL$. One can check as in \cite[Lemma 2.2]{LPcoupled} that
\begin{equation}\label{eq:proj.commute}
J\circ\pi_L=\pi_J\circ J\quad\text{and}\quad J\circ\pi_J=\pi_L\circ J.
\end{equation}

Choose $\Omega$ as above. Given vectors $v_1,\dots,v_n\in V$, it is simple to verify that the following conditions are equivalent: (i) the vectors are a complex basis; (ii) the real span of the vectors is a totally real subspace; (iii) $\Omega(v_1,\dots,v_n)\neq 0$.

We say that $\pi$ is $\Omega$-\textit{special totally real} (STR) if $\Omega(v_1,\dots,v_n)\in \R^+$. The basis then defines a canonical orientation for $\pi$.

Notice that the space of complex volume forms on $V$ is parametrized by the unit circle $\Sph^1$. The space of oriented totally real planes is parametrized by $\GL(n,\C)/\GL^+(n,\R)$. Finally, the space of oriented $\Omega$-STR planes is parametrized by $\{M\in \GL(n,\C): \det_\C M\in \R^+\}$, modulo $\GL^+(n,\R)$.

There exists an interesting converse to this set-up, which furnishes a key ingredient to the geometry of totally real submanifolds. Start by choosing an oriented totally real plane $\pi$. Let $v_1,\dots,v_n$ be any positive basis of $\pi$: this generates a real basis $v_1,\dots,v_n,Jv_1,\dots,Jv_n$ of $V$, thus a dual real basis of $V^*$. Set
\begin{equation*}
\Omega:=\frac{(v_1^*+i(Jv_1)^*)\wedge\dots\wedge(v_n^*+i(Jv_n)^*)}{|(v_1^*+i(Jv_1)^*)\wedge\dots\wedge(v_n^*+i(Jv_n)^*)|_h}.
\end{equation*}
It is elementary to check that $\Omega$ is independent of the chosen basis,  so it is a complex volume form on $V$ canonically defined by $\pi$, with respect to which $\pi$ is STR. We remark that if we start with $\Omega$ and choose an STR $\pi$, this construction yields a complex volume which is necessarily of the form $e^{i\theta}\cdot\Omega$, for some $\theta$. Since it takes positive real values on $\pi$, it must be that $e^{i\theta}=1$ so we have actually recovered the original $\Omega$. 

Notice that, up to here, we have only used the induced Hermitian metric on $\Lambda^{(n,0)}V^*$. The full metric on $V$ allows us two further options. First, it allows us to define the standard volume form on $\pi$,
\begin{equation*}
 \vol_g[\pi]:=e_1^*\wedge\dots\wedge e_n^*,
\end{equation*}
where $e_1,\dots,e_n$ is a positive $g$-orthonormal (ON) basis of $\pi$. Let $\vol_J[\pi]$ denote the restriction of $\Omega$ to $\pi$. Both $\vol_g[\pi]$ and $\vol_J[\pi]$ are real volume forms on $\pi$, so we can compare them. Writing
\begin{equation}\label{eq:canonical_vol}
\Omega=\frac{(e_1^*+i(Je_1)^*)\wedge\dots\wedge(e_n^*+i(Je_n)^*)}{|(e_1^*+i(Je_1)^*)\wedge\dots\wedge(e_n^*+i(Je_n)^*)|_h},
\end{equation}
one can check that $\vol_J[\pi]=\rho_J \cdot \vol_g[\pi]$, where
\begin{align}
\rho_J:&=|(e_1^*+i(Je_1)^*)\w\ldots\w (e_n^*+i(Je_n)^*)|_h^{-1}\nonumber\\
&=({\det}_{\C}h_{ij})^{\frac{1}{2}}=\sqrt{\vol_{\bar g} (e_1,\dots,e_n,Je_1,\dots,Je_n)},\label{h.mod.eq}
\end{align}
with $h_{ij}=h(e_i,e_j)$.

Second, the Hermitian metric on $V$ allows us to define a particular subclass of totally real planes via the \textit{Lagrangian} condition $\bar\omega_{|\pi}=0$; equivalently, $J\pi$ is orthogonal to $\pi$. One can check that $|\Omega(e_1,\dots,e_n)|\leq 1$, equivalently $\rho_J\leq 1$, with equality if and only if $\pi$ is Lagrangian. 

The $\Omega$-STR planes which are also Lagrangian are known in the literature as \textit{special Lagrangian}.

\paragraph{Differentiation of complex volume forms.}
Let $(M,J,h=\bar{g}-i\bar{\omega})$ be a Hermitian manifold endowed with a Hermitian connection $\tnabla$ with torsion tensor $\widetilde{T}$. Recall that $M$ is \textit{almost K\"ahler} if $\bar{\omega}$ is closed, 
thus $M$ is symplectic; in this case we will always use the \textit{Chern connection}, whose torsion tensor $\widetilde{T}$ is of type $(1,1)$.  
We have that $M$ is \textit{K\"ahler} if $\bar{\omega}$ is closed and $J$ is integrable; in this case the Chern connection is the Levi-Civita connection $\overline{\nabla}$ of $\bar{g}$, which is torsion-free.

It follows from Chern--Weil theory that the first Chern class of $M$ can be represented in terms of curvature. Specifically,
for $X,Y\in T_pM$ we set 
\begin{equation}\label{P.eq}
\widetilde{P}(X,Y):=\overline{\omega}(\widetilde{R}(X,Y)\overline{e}_j,\overline{e}_j),
\end{equation}
where $\widetilde{R}$ is the curvature of $\tnabla$ and $\overline{e}_1,\ldots,\overline{e}_{2n}$ is an orthonormal basis for $T_pM$.  Then $\widetilde{P}$ is a closed $2$-form on $M$, 
$-\frac{i}{2}\widetilde{P}$ is the curvature of $K_M^*$ and
$$2\pi c_1(M)=\left[\frac{1}{2}\widetilde{P}\right].$$ 
In the K\"ahler case, defining the Ricci form $\bar{\rho}(\cdot,\cdot):=\overline{\Ric}(J\cdot,\cdot)$, one finds
\begin{equation}\label{P.rho.eq}
\frac{1}{2}\widetilde{P}(X,Y)=\bar{\rho}(X,Y).
\end{equation}

Let $\gamma=\gamma(t)$ be a curve in $M$ with $\dot{\gamma}(0)=Z\in T_pM$. Assume we are given a complex volume form $\Omega(t)$ over $\gamma(t)$, \textit{i.e.}~a smooth unit section of the canonical bundle $K_{M|\gamma}$. We can calculate $\widetilde{\nabla}_Z\Omega$ using the following formula which extends \cite[Proposition 4.3]{LPcoupled}, concerning curves of complex volume forms generated by submanifolds, to the arbitrary curves of complex volume forms of interest here. The proof is similar, so we omit it.
\begin{prop}\label{prop:diff_Omega} 
Choose any family of totally real planes $\pi(t)$ over $\gamma(t)$ such that each $\pi(t)$ is $\Omega(t)$-STR. Let $\pi_J(t)$ denote the corresponding projections and $e_1(t),\dots, e_n(t)$ be any positive $g(t)$-orthonormal basis of $\pi(t)$, where $g(t)$ is the induced metric on $\pi(t)$. Then,
\begin{equation*}
\widetilde\nabla_Z\Omega=i\,g(J\pi_J\widetilde\nabla_Ze_i,e_i)\cdot\Omega.
\end{equation*}
\end{prop}

Proposition \ref{prop:diff_Omega} introduces the quantity $g(J\pi_J\widetilde\nabla_Ze_i,e_i)$ and gives indirect evidence that it is independent of the particular choices made. This is not immediately obvious, so we give a direct proof of this fact.

\begin{lem} Let $\Omega(t)$ be a family of complex volume forms on $M$, as above. Choose any two families $\pi(t)$, $\pi'(t)$ of $\Omega(t)$-STR planes. 
Let $g$, $g'$ and $\pi_J$, $\pi_J'$ denote the corresponding metrics and projections. Choose any two families of positive ON frames $e_i(t)$, $e_i'(t)$ for $\pi(t),\pi'(t)$ respectively. Then
\begin{equation*}
 g(J\pi_J\widetilde\nabla_Ze_i,e_i)=g'(J\pi_J'\widetilde\nabla_Ze_i',e_i').
\end{equation*}
\end{lem}

\begin{proof} Using the parametrization of STR planes given above one can show that $e_i'(t)=A(t)e_i(t)$ for some $A(t)\in\GL(T_pM,\C)$ such that, for all $t$, $\det_\C A\in\R^+$ and $\bar g(Ae_i,Ae_j)=\delta_{ij}$. 
Furthermore, $\pi_J'=A\pi_JA^{-1}$. Thus
\begin{align*}
g'(J\pi_J'\widetilde\nabla_Ze_i',e_i')&=\bar g(JA\pi_JA^{-1}\widetilde\nabla_Z(Ae_i),Ae_i)\\
&=\bar g(AJ\pi_JA^{-1}((\widetilde\nabla_ZA)e_i+A(\widetilde{\nabla}_Ze_i)),Ae_i)\\
&=\bar g(J\pi_JA^{-1}(\widetilde\nabla_ZA)e_i, e_i)+\bar g(J\pi_J\widetilde{\nabla}_Ze_i,e_i).
\end{align*}
It now suffices to show that $\bar g(J\pi_JA^{-1}(\widetilde\nabla_ZA)e_i, e_i)=0$.

Using a parallel basis along $\gamma$ we can identify $A$ with a family of matrices in $GL(n,\C)$; then $\tnabla_ZA=\frac{d}{dt}A$. A standard formula for matrices shows that 
$$\tnabla_Z\textrm{det}_\C A=\tr_\C(A^{-1}\tnabla_ZA)=\bar{g}(\pi_LBe_i,e_i)-i\bar{\omega}(\pi_JBe_i,e_i),$$ where we set $B:=A^{-1}\tnabla_ZA$ and the projections are included since the basis given by 
$\{e_1,\ldots,e_n,Je_1,\ldots,Je_n\}$ is not orthogonal.

Since $\det_\C A$ is real it follows that $\bar\omega(\pi_J Be_i,e_i)=0$, and thus we have that $\bar g(J\pi_J Be_i,e_i)=0$.
\end{proof}

In other words, the quantity at hand, $g(J\pi_J\widetilde\nabla_Ze_i,e_i)$, depends only on the STR equivalence classes of the family of planes.
\paragraph{Totally real submanifolds.}
Let $M$ be a Hermitian manifold. 
Let $L$ be a smooth compact oriented $n$-dimensional manifold. An immersion $\iota:L\to M$ is \textit{totally real} if each tangent space $\iota_*(T_pL)$ is totally real in $T_{\iota(p)}M$. We will often identify $L$ with its image in $M$. With this notation one obtains a splitting
\begin{equation}\label{eq:splitting}
 T_pM=T_pL\oplus J(T_pL)
\end{equation}
with projections $\pi_L:T_pM\to T_pM$ and $\pi_J:T_pM\to T_pM$, as above. 

Together with the orientability assumption on $L$, the splitting \eqref{eq:splitting} implies that the pull-back bundle $K_M[\iota]:=\iota^*K_M$ over $L$ is trivial. Applying (\ref{eq:canonical_vol}) to each $T_pL$ we obtain a canonical section $\Omega_J[\iota]$ of $K_M[\iota]$.
As explained, $L$ is automatically $\Omega_J[\iota]$-STR so $\vol_J[\iota]:=\iota^*(\Omega_J[\iota])$ defines a \textit{real} volume form on $L$. Specifically,
\begin{equation}\label{eq:volJ_vs_vol}
\vol_J[\iota]=\rho_J\cdot\vol_g[\iota], 
\end{equation}
where $\rho_J$ is defined as in (\ref{h.mod.eq}) and $\vol_g[\iota]$ is the standard volume form on $L$ defined by $\iota$. 

We can use Proposition \ref{prop:diff_Omega} to calculate the derivatives $\widetilde\nabla_Z\Omega_J[\iota]$. By linearity it suffices to consider the two cases $Z:=X$ and $Z:=JX$, for some $X$ tangent to $L$.

Consider the case when $Z:=X$ is tangent to $L$. One may easily check that, given a locally defined tangent vector field $v$ and function $f$ on $L$, $J\pi_J\widetilde\nabla_X(fv)=f\cdot J\pi_J\widetilde\nabla_X v$. Using this fact one can show that
\begin{equation}
 v\in T_pL\rightarrow J\pi_J\widetilde\nabla_X v\in T_pL
\end{equation}
is a well-defined endomorphism of $T_pL$. Let $\xi_J[\iota](X)$ denote its trace, so that $\xi_J[\iota]$ is a 1-form on $L$. It follows from Proposition \ref{prop:diff_Omega} that 
\begin{equation}\label{eq:nabla_Omega}
 \widetilde\nabla_X\Omega_J[\iota]=i\,\xi_J[\iota](X)\cdot\Omega_J[\iota],
\end{equation}
showing that $\xi_J[\iota]$ is the \textit{connection 1-form} associated to the section $\Omega_J[\iota]$ of the complex line bundle $K_M[\iota]$ over $L$. In \cite{LPcoupled} this is called the \textit{Maslov 1-form} of $\iota$; its role with respect to Maslov index theory and holomorphic curves with boundary on $L$ is explained in \cite{PMaslov}. 

Standard theory shows that $\d(i\,\xi_J[\iota])$ is the curvature of the line bundle $K_M[\iota]$, thus
\begin{equation}\label{eq:dxi=ric}
 \d\xi_J[\iota]=\frac{1}{2}\iota^*\widetilde{P}.
\end{equation}

Now assume given a 1-parameter family $\iota_t$ of totally real immersions. Let $Z:=\frac{\partial\iota_t}{\partial t}_{|t=0}$. Assume $Z$ is of the form $Z=JX$, for some vector field tangent to $\iota_0$. 
The proof of \cite[Proposition 5.3]{LPtotally} shows that 
$$g(J\pi_J\widetilde\nabla_{JX}e_i,e_i)=-\frac{\Div(\rho_J X)}{\rho_J}+2g(J\pi_J\widetilde{T}(JX,e_i),e_i).$$ It follows from Proposition \ref{prop:diff_Omega} that, in the K\"ahler setting, 
\begin{equation}
 \widetilde\nabla_{JX}\Omega_J[\iota]=-i\frac{\Div(\rho_J X)}{\rho_J}
 \cdot\Omega_J[\iota].
\end{equation}

\paragraph{The J-volume functional.} Let $M$ be a Hermitian manifold. Recall the Riemannian volume functional on the space of all immersions $\iota:L\rightarrow M$, defined by integrating the standard volume form: $\Vol_g(\iota):=\int_L\vol_g[\iota]$.

The space of totally real immersions is an open subset of the space of all immersions. Restricting to this domain, we can alternatively consider the \textit{$J$-volume} functional $\Vol_J(\iota):=\int_L\vol_J[\iota]$. One should think of $\Vol_J$ as a modified volume which takes into account the totally real condition.

It is interesting to compare these two functionals on their common domain. In particular, it follows from (\ref{h.mod.eq}) and (\ref{eq:volJ_vs_vol}) that $\Vol_J$ furnishes a lower bound for $\Vol_g$: $\Vol_J(\iota)\leq \Vol_g(\iota)$, with equality if and only if $\iota$ is \textit{Lagrangian}, \textit{i.e.}~$\iota^*\overline\omega=0$.

In general, the volume and $J$-volume functionals on totally real immersions will have different critical points. However, in \cite{LPcoupled} we show that if $\iota$ is Lagrangian then 
the standard volume and the $J$-volume agree to first order, \textit{i.e.}~the corresponding gradients also coincide. 

When $M$ is K\"ahler we can go further. Let $H_J[\iota]$ denote the \textit{$J$-mean curvature} vector field, \textit{i.e.}~the gradient of the $J$-volume functional. In \cite[Theorem 5.2]{LPcoupled} we show that $H_J$ 
and $\xi_J$ basically coincide: specifically, they are related by the formula
\begin{equation}\label{eq:xi=H}
 H_J=-J\iota_*(\xi_J^\#).
\end{equation}
More generally, when $M$ is almost K\"ahler then the gradient of $\Vol_J$ incorporates terms generated by the torsion of the Chern connection, and (\ref{eq:xi=H}) holds only up to torsion corrections. 

\begin{remark}\label{rem:calibration}
It is interesting to compare $\Vol_J$ and $\Vol_g$ also from other points of view. The main feature of totally real submanifolds is that, through the isomorphism $TL\simeq J(TL)$, 
they manage to relate extrinsic information regarding $TM_{|L}$ to intrinsic information. Two manifestations of this are as follows.
\begin{itemize}
 \item The form $\vol_J[\iota]$ extends from $T_pL$ to $T_pM$, via $\Omega_J[\iota]$: this is similar in spirit (though weaker, because it is  localized to $L$) to the situation of calibrated geometry, where the calibrating form offers a global extension of the standard volume form on a calibrated submanifold. This point of view is further developed in Section \ref{s:grassmannian}, then applied in Section \ref{s:linearisation} to obtain a simple proof of the linearisation formula for the Maslov form.
 \item Replacing $\vol_g$ with $\vol_J$ implies substituting a section of the real line bundle of volume forms with a section of a complex line bundle: this generates extra geometry via the curvature of the complex line bundle, which is ultimately related to the Ricci curvature of $M$.
\end{itemize}
Another manifestation appears in the notion of geodesics in the space of totally real submanifolds, introduced in \cite{LPtotally}.
\end{remark}

\paragraph{$J$-minimal immersions.}
Let $M$ be a K\"ahler manifold. 
Recall that an immersion is \textit{minimal} if it is a critical point of the standard volume functional, \textit{i.e.}~if the mean curvature vector field $H$ vanishes. We say that a totally real immersion 
is \textit{$J$-minimal} if it is a critical point of the $J$-volume functional.

Equations (\ref{eq:nabla_Omega}) and (\ref{eq:xi=H}) provide an alternative interesting geometric interpretation of the $J$-minimal immersions: these are the totally real immersions $\iota$ for which $\xi_J[\iota]=0$, 
\textit{i.e.}~$\Omega_J[\iota]$ is parallel. This condition is also notable for the following reasons.

First, suppose that $\bar\rho$ has a definite sign, \textit{i.e.}~$\pm\overline\Ric$ is a Riemannian metric. This forces $c_1(M)$ to have a definite sign. Since $\bar\rho$ is closed and of type $(1,1)$, 
it defines a symplectic (actually K\"ahler) structure on $M$. Using (\ref{P.rho.eq}) and (\ref{eq:dxi=ric}) we see that $\xi_J[\iota]=0$ implies that $\iota^*\bar\rho=\d\xi_J[\iota]=0$, so the submanifold is 
 $\bar\rho$-Lagrangian. Notice that any $\bar\rho$-Lagrangian is automatically totally real because $T_pL$ and $J(T_pL)$ are orthogonal with respect to the metric $\pm\overline{\Ric}$.

Now assume $M$ is K\"ahler--Einstein with non-zero scalar curvature. In this case the two K\"ahler structures given by $\bar\omega$ and $\bar\rho$ are equal (up to a multiplicative constant) so the submanifold is also Lagrangian in the standard sense. In \cite[Proposition 5.3]{LPcoupled} we showed that the converse also holds: the only critical points of the $J$-volume are the minimal Lagrangian submanifolds. 
We summarize as follows.

\begin{thm}\label{thm:minlag_is_Jmin}
Let $M$ be a K\"ahler manifold and $\iota:L\to M$ a totally real immersion. The following conditions are equivalent:
\begin{itemize}
\item $\iota$ is $J$-minimal;
\item $H_J[\iota]=0$;
\item $\xi_J[\iota]=0$;
\item $\Omega_J[\iota]$ is parallel.
\end{itemize}
If $M$ is K\"ahler-Einstein with non-zero scalar curvature then these conditions are also equivalent to $\iota$ being minimal Lagrangian. More generally, if $M$ has definite Ricci form then these conditions imply that $\iota$ is $\bar\rho$-Lagrangian.
\end{thm}
In \cite{LPtotally}, when $M$ is negative K\"ahler--Einstein, we obtain a further characterization of minimal Lagrangians as the zero set of a moment map.

Notice from Theorem \ref{thm:minlag_is_Jmin} that, in the KE case, $H_J$ offers simultaneous control over both the minimal and the Lagrangian condition, which need to be studied separately if one works with the standard mean curvature $H$.

\section{The totally real Grassmannian}\label{s:grassmannian}

Let $M$ be a Hermitian manifold. In general $M$ does not admit a global complex volume form, \textit{i.e.}~the canonical bundle $K_M$ does not admit smooth non-zero sections: 
 this would imply that $K_M$ is differentiably trivial, thus $c_1(M)=0$.

It is thus an interesting fact that $K_M$ is trivial when restricted to totally real submanifolds.

We can trivialize $K_M$ globally by lifting it to a Grassmannian bundle. Specifically, consider the Grassmannian $TR^+(M)$ of oriented totally real $n$-planes in $TM$. Using the projection $q:TR^+(M)\to M$ we can pull back the bundle $K_M$ together with its Hermitian metric and connection $\tnabla$. The corresponding curvature of this connection is then the 
pull-back tensor $q^*(\frac{i}{2}\widetilde{P})$. We thus obtain a complex line bundle over $TR^+(M)$ which, by the same construction as above, admits a global
canonical section $\Omega$. Specifically, given a totally real plane $\pi\in TR^+(M)$ and a positive orthonormal basis $e_1,\ldots,e_n$ for $\pi$, we define $\Omega(\pi)$ as in \eqref{eq:canonical_vol}. 
In particular, $q^*K_M$ is differentiably trivial for any $M$.

We now want to develop a formula for the derivatives of $\Omega$, analogous to Proposition \ref{prop:diff_Omega}.
Let $\pi=\pi(t)$ be a curve in $TR^+(M)$ with $\dot\pi(0)=\hat{Z}\in T_{\pi(0)}TR^+(M)$. Let $\gamma(t):=q\circ\pi(t)$ and $Z:=q_*\hat{Z}$ be the corresponding data on $M$. By definition of pull-back, $q^*(K_M)_{|\pi(t)}$ 
can be canonically identified with $K_{M|\gamma(t)}$; under this identification, $\Omega(t)$ corresponds to a section $\Omega_J(t)$ of $K_M$ defined along $\gamma(t)$.

With this notation, by definition of the pull-back connection, we have
\begin{equation}\label{eq:same_nabla}
\tnabla_{\hat{Z}}\Omega=\tnabla_Z\Omega_J.
\end{equation}

Let us define a $1$-form $\Xi$ on $TR^+(M)$ as follows: 
\begin{equation}\label{Xi.eq}
\Xi_{\pi}(\hat{Z}):=g(J\pi_J\tnabla_Ze_i,e_i),
\end{equation}
 where $\pi_J$ is the projection defined by $\pi$ and the $e_i$ define an ON basis of $\pi$. It follows from (\ref{eq:same_nabla}) and from Proposition \ref{prop:diff_Omega} that
\begin{equation}\label{Xi.conn.eq}
\tnabla\Omega=i\,\Xi\otimes\Omega,
\end{equation}
so that $\d\Xi=q^*\left(\frac{1}{2}\widetilde{P}\right)$. If $\hZ$, $\hW$ are vector fields on $TR^+(M)$ with projections $Z:=q_*(\hZ), W:=q_*(\hW)$ in $TM$, a standard formula for the differential of a $1$-form then shows that
\begin{equation*}
 \hZ(\Xi(\hW))-\hW(\Xi(\hZ))-\Xi([\hZ,\hW])=\d\Xi(\hZ,\hW)=q^*\left(\frac{1}{2}\widetilde{P}(Z,W)\right).
\end{equation*}

We can relate $\Xi$ and $\Omega$ to the corresponding data on totally real submanifolds as follows. Let $\iota:L\to M$ be totally real. Let 
\begin{equation*}
\hiota:L\to TR^+(M),\ \ p\mapsto \iota_*(T_pL)                                                                                                   
\end{equation*}
denote the Gauss map of $\iota$, so that $\iota=q\circ\hiota$.  Then $\Omega_J[\iota]=\Omega\circ\hiota$ and $\xi_J[\iota]=\hiota^*\Xi$.

\begin{remark}\label{rem:Iorder}
 Notice that $\hiota$ is of first order with respect to $\iota$. Likewise, given a curve of immersions $\iota_t$ and the vector field $Z:=\frac{\partial}{\partial t}\iota$, 
the lifted vector field $\hZ=\frac{\partial}{\partial t}\hiota$ is of first order with respect to $Z$.
\end{remark}

\section{Linearisation of the Maslov map}\label{s:linearisation}

Fix an initial totally real immersion $\bar{\iota}:L\to M$ into a Hermitian manifold $M$. Let $\mathcal{P}$ be the space of totally real immersions $\iota$ of $L$ in $M$ isotopic to $\bar{\iota}$. Formally, 
 this is an infinite-dimensional manifold with tangent space $T_{\iota}\mathcal{P}=\Lambda^0(\iota^*TM)$. The term $T_pL$ in the splitting \eqref{eq:splitting} corresponds to the infinitesimal action determined by reparametrisation. 
 Thus, $\mathcal{P}$ can be viewed as a $\mbox{Diff}(L)$-principal fibre bundle over the quotient space $\mathcal{T}$ of non-parametrised totally real submanifolds in $M$ isotopic to $\bar{\iota}(L)$.

Consider the \textit{Maslov map} from $\mathcal{P}$ to the 1-forms $\Lambda^1(L)$ on $L$, 
\begin{equation}\label{eq:Maslov.map}
\xi_J:\mathcal{P}\to \Lambda^1(L)
\end{equation} 
given by $\iota\mapsto \xi_J[\iota]$. Our first goal is to calculate its differential.

Let $D\xi_{J|\iota}:T_\iota\mathcal{P}\to\Lambda^1(L)$ denote the differential of the Maslov map at the point $\iota\in\mathcal{P}$ and let $Z$ denote a vector in $T_\iota\mathcal{P}$. 
As seen above we can view $Z$ as a section of $\iota^*(TM)$, \textit{i.e.}~as a vector field on $M$ defined along the submanifold $\iota(L)$. Let $\iota_t$ be a curve of totally real immersions with $\iota_0=\iota$ 
 such that $\frac{\partial}{\partial t}\iota_{|t=0}=Z$.  By definition
\begin{equation}\label{eq:Dxi}
D\xi_{J|\iota}(Z)=\frac{d}{dt}(\xi_J[\iota_t])_{|t=0},
\end{equation}
where at each point $p\in L$ we are differentiating the curve $\xi_J[\iota_t]_{|p}$ in the fixed vector space $T_p^*L$.

Using the notation of Section \ref{s:grassmannian} we can  calculate the right-hand side of \eqref{eq:Dxi} by lifting it into $TR^+(M)$:
\begin{align}
 D\xi_{J|\iota}(Z)&=\frac{d}{dt}\,(\hiota_t^*\Xi)_{|t=0}\nonumber\\
 &=\hiota^*(\mathcal{L}_{\hZ}\Xi)\nonumber\\
 &=\hiota^*(\d(\hZ\lrcorner\Xi)+\hZ\lrcorner \d\Xi)\nonumber\\
 &=\hiota^*\d(\Xi(\hZ))+\frac{1}{2}(q\circ\hiota)^*\widetilde{P}(Z,\cdot).
 \label{eq:Dxi.2}
 \end{align}

Notice that using the splitting \eqref{eq:splitting} of $TM$ we can write $Z=X+JY$, where $X,Y\in \Lambda^0(TL)$. As usual this relies on identifications: more precisely, $Z=\iota_*(X)+J\iota_*(Y)$. 

If the transverse component $JY$ of $Z$ vanishes, \textit{i.e.}~$Z=X$, then by commuting $\hiota^*$ and the exterior derivative $\d$ we see that \eqref{eq:Dxi.2} gives
\begin{equation}\label{eq:Dxi.vert}
D\xi_{J|\iota}(X)=\d(\xi_J[\iota](X))+\frac{1}{2}\iota^*\widetilde{P}(X,\cdot).
\end{equation}
We could have alternatively found this formula directly from   \eqref{eq:Dxi} by noticing that $X$ is ``vertical'' with respect to the $\mbox{Diff}(L)$-action on $\mathcal{P}$: 
if we integrate $X$ to a flow $\phi_t$ on $L$ and set $\iota_t:=\iota\circ\phi_t$ then, using the fact that $\xi_J[\iota_t]=\phi_t^*(\xi_J[\iota])$ and \eqref{eq:dxi=ric}, we obtain \eqref{eq:Dxi.vert}.

In particular, if $\xi_J[\iota]=0$ then $\d\xi_J[\iota]=\frac{1}{2}\iota^*\widetilde{P}=0$ so $D\xi_{J|\iota}(X)=0$: this is a manifestation of the $\mbox{Diff}(L)$-invariance of the condition $\xi_J=0$.

We now consider the case $Z=JY$ and obtain the following.

\begin{lem}\label{diff.xi.prop}
Assume $M$ is almost K\"ahler. Let $X,Y$ be tangent vector fields on a totally real submanifold $\iota:L\to M$ and let $\iota_t:L\to M$ be a family of totally real immersions such that $\iota_0=\iota$ 
and $\frac{\partial}{\partial t}{\iota_t}_{|t=0}=JY$.  Then
\begin{align*}
D\xi_{J|\iota}(JY)(X)&=-\tnabla_X\left(\frac{\Div(\rho_JY)}{\rho_J}-\bar{g}(\pi_L\widetilde{T}(Y,e_j),e_j)\right)+\frac{1}{2}\widetilde{P}(JY,X).
\end{align*}
When $M$ is K\"ahler we have:
\begin{align*}
D\xi_{J|\iota}(JY)(X)&=-\overline{\nabla}_X\left(\frac{\Div(\rho_JY)}{\rho_J}\right)-\overline{\Ric}(X,Y).
\end{align*}
\end{lem}
\begin{proof}  Let $Z=JY$.
From \eqref{Xi.eq} we know that, at $p\in L$, $\Xi(\hZ)=g(J\pi_J\tnabla_{JY}e_j,e_j)$ where $e_1,\ldots,e_n$ is an orthonormal basis for $T_pL$.

Let $v_j$ for $j=1,\ldots,n$ be any smooth local vector fields on $L$ extending the vectors $e_j$ and set $w_j=\iota_{t*}(v_j)$. Then $[\frac{\partial}{\partial t},v_j]=0$ locally on $L\times\R$, so at $t=0$ 
 we have $[JY,w_j]=\iota_*[\frac{\partial}{\partial t}, v_j]=0$.
Then, using \eqref{eq:proj.commute} and the fact that $\widetilde{T}(JY,w_j)=-J\widetilde{T}(Y,w_j)$ for the Chern connection (see, for example, \cite[Section 4.3]{LPcoupled}), we have at $t=0$:
\begin{align}
 \bar{g}(J\pi_J\tnabla_{JY}w_j,w_j)&=\bar{g}(J\pi_J(\tnabla_{w_j}JY+\widetilde{T}(JY,w_j)),w_j)\nonumber\\
 &=\bar{g}(J\pi_{J}J\tnabla_{w_j}Y,w_j)-\bar{g}(J\pi_JJ\widetilde{T}(Y,w_j),w_j)\nonumber\\
 &=-\bar{g}(\pi_L\tnabla_{w_j}Y,w_j)+\bar{g}(\pi_L\widetilde{T}(Y,w_j),w_j).\label{Xi.JY.eq0}
\end{align}
Evaluating \eqref{Xi.JY.eq0} at $p$ and using calculations in \cite[Proposition 5.3]{LPtotally} we find
\begin{equation}
\Xi(\hZ)=-\frac{\Div(\rho_JY)}{\rho_J}+\bar{g}(\pi_L\widetilde{T}(Y,e_j),e_j)
.\label{Xi.JY.eq}
\end{equation}
We emphasize that both sides of \eqref{Xi.JY.eq} are of first order with respect to $Y$, cf.~Remark \ref{rem:Iorder}. The first result follows from \eqref{eq:Dxi.2}.

In the K\"ahler case, $\tnabla=\overline{\nabla}$ so the torsion term vanishes and using \eqref{P.rho.eq} we have
$$\frac{1}{2}\widetilde{P}(JY,X)=\overline{\Ric}(J(JY),X)=-\overline{\Ric}(Y,X)=-\overline{\Ric}(X,Y),$$
completing the proof.
\end{proof}

In particular we have proved the following.
\begin{prop} \label{lin.xi.prop}
Assume $M$ is K\"ahler, $\iota:L\to M$ is totally real and $\xi_J[\iota]=0$. The linearisation of the Maslov map \eqref{eq:Maslov.map} 
at $\iota$ is given by
$$D\xi_{J|\iota}(X+JY)=\d(\rho_J^{-1}\d^*(\rho_JY^{\flat}))-\iota^*\overline{\Ric}(Y,.)$$
for $X,Y\in\Lambda^0(TL)$.
In particular $\{X\in \Lambda^0(TL)\}\subseteq\ker (D\xi_{J|\iota})$.  
\end{prop}

\begin{remark}\label{rem:convexity}
In order to link these results on K\"ahler manifolds to the second variation formula obtained in \cite[Proposition 5.8]{LPtotally}, let $\iota_t$ be a curve of immersions such that $\iota_0=\iota$ and $\frac{\partial\iota}{\partial t}=JY$, for some vector field $Y=Y(t)$ on $L$. Then
\begin{align}
\frac{d}{dt}\Vol_J(\iota_t)&=\int_L\frac{\partial}{\partial t}\vol_J[\iota_t]\nonumber\\
&=-\int_L\bar g(H_J[\iota_t],JY)\vol_J[\iota_t]=\int_L\xi_J[\iota_t](Y)\vol_J[\iota_t],\label{firstvar.eq}\\
\frac{d^2}{dt^2}\Vol_J(\iota_t)=&\int_L\left(\frac{\partial}{\partial t}\xi_J[\iota_t]\right)(Y)\vol_J[\iota_t]+\int_L\xi_J[\iota_t]\left(\frac{\partial}{\partial t}Y\right)\vol_J[\iota_t]\nonumber\\
&+\int_L\xi_J[\iota_t](Y)\frac{\partial}{\partial t}\vol_J[\iota_t].\label{secondvar.eq}
\end{align}
Let us examine the three terms on the right-hand side of \eqref{secondvar.eq}. Using Lemma \ref{diff.xi.prop} one can show that the first term coincides with
$$\int_LD\xi_J(Y)\vol_J=\int_L\left(\frac{\Div(\rho_JY)}{\rho_J}\right)^2\vol_J-\int_L\overline{\Ric}(Y,Y)\vol_J.$$
The second term can be identified with 
$$\int_L\xi_J(\pi_L[JY,Y])\vol_J=-\int_L\bar g(H_J,\pi_JJ[JY,Y])\vol_J.$$
From \eqref{firstvar.eq}, we see that the third term is
\begin{align*}
 \int_L\big(\xi_J(Y)\big)^2\vol_J.
\end{align*}
 When $\overline{\Ric}\leq 0$ all terms are non-negative except the second, which however vanishes when $\iota_t$ is a ``geodesic'' in the sense of \cite{LPtotally}. This gives a new proof of the convexity result \cite[Theorem 5.10]{LPtotally}, also showing the efficiency of the techniques introduced in this paper.

\subsection{Restriction to Lagrangian deformations}\label{ss:weinstein}
\end{remark}

Let $M$ be a K\"ahler manifold with definite Ricci tensor. Theorem \ref{thm:minlag_is_Jmin} shows that any $J$-minimal submanifold is Lagrangian with respect to the appropriate symplectic form. This is a strong condition, 
 with notable consequences. We are particularly interested in the \textit{Lagrangian neighbourhood theorem} applied to any initial compact Lagrangian immersion $\iota:L\to M$ in any symplectic manifold, as follows. 
To simplify notation we will identify $L$ with its image $\iota(L)$.

Recall that $N:=T^*L$ admits a standard symplectic structure such that Lagrangian sections are precisely the graphs of closed 1-forms. The Lagrangian neighbourhood theorem states that there is an open neighbourhood $V$ of the zero section in $N$, a tubular neighbourhood $T$ of $L$ in $M$ 
and a symplectomorphism  $\Phi:V\to T$ which is the identity on $L$. This construction depends on an initial choice of a Lagrangian distribution $E\leq TM_{|L}$, transverse to $TL$: one then obtains that, for each $p\in L$, the differential map 
\begin{equation}\label{eq:weinstein}
 \d\Phi_p: T_pN= T_pL\oplus T^*_pL\to T_pM=T_pL\oplus E_p
\end{equation}
provides an isomorphism between the corresponding subspaces in this splitting. Notice that each such subspace is Lagrangian. 

The main application of this theorem is the following. Let $\alpha\in \Lambda^1(L)$ be  such that $\Graph(\alpha)\subseteq V$ and let $\iota_{\alpha}:=\Phi\circ\alpha:L\to M$ denote the corresponding immersion. Then $\iota_{\alpha}:L\to M$ is Lagrangian if and only if $\d\alpha=0$. Notice that this result allows us (i) to ``gauge-fix'' the space of 
Lagrangian immersions, \textit{i.e.}~to locally eliminate the role of the reparametrization group $\Diff(L)$, by providing a canonical parametrization of the nearby Lagrangian submanifolds, and (ii) to ``linearize'' this space of Lagrangian submanifolds via the vector space of closed 1-forms.

We can apply the Lagrangian neighbourhood theorem to any Lagrangian immersion in the symplectic manifold $(M,\bar\rho)$ defined by the Ricci 2-form. To this end we choose the Lagrangian distribution $E_p:=J(T_pL)$. For any $k\geq 0$ and $a\in(0,1)$ let 
\begin{equation}\label{closed.forms.eq}
\mathcal{Z}^{k+2,a}=\{\alpha\in C^{k+2,a}(T^*L)\,:\,\d\alpha=0\}
\end{equation}
denote the closed 1-forms on $L$ in $C^{k+2,a}$. Let $\mathcal{U}$ be the subset of forms in $\mathcal{Z}^{k+2,a}$ whose graph lies in $V$: these forms parametrize the $\bar\rho$-Lagrangians near $\iota(L)$ which have the corresponding degree of smoothness. When $M$ is KE (with non-zero scalar curvature) we recover the usual Lagrangian submanifolds, defined with respect to $\bar\omega$.

Consider the restricted Maslov map
\begin{equation}\label{F.eq}
F:\mathcal{U}\to \mathcal{Z}^{k,a},\ \ F(\alpha)=\xi_J[\iota_{\alpha}].
\end{equation}
This map is well-defined between these spaces because $\xi_J$ is second order in $\alpha$ and  each form $F(\alpha)$ is closed: 
$\d\xi_J[\iota_{\alpha}]=\iota_{\alpha}^*\overline{\rho}=0$, where we use the fact that $\iota_\alpha$ is $\bar\rho$-Lagrangian.

We now notice the following fact.

\begin{lem}\label{l:exact}
 Assume $\xi_J[\iota]$ is exact. In terms of the Lagrangian neighbourhood theorem applied to $\iota$, if $\alpha$ is exact then $\xi_J[\iota_\alpha]$ is exact.
\end{lem}
\begin{proof}
 Assume $\alpha=\d f$. It suffices to prove that, for any closed curve $\gamma$ in $L$, $\int_\gamma\xi_J[\iota_\alpha]=0$. Consider the curve of immersions $\iota_t=\iota_{t\alpha}$. Then
 \begin{align*}
  \frac{d}{dt}\int_\gamma \xi_J[\iota_t]&=\int_\gamma\frac{d}{dt}\hat\iota_t^*\Xi=\int_\gamma\hat\iota_t^*\mathcal{L}_{\frac{\partial \hat\iota_t}{\partial t}}\Xi\\
  &=\int_\gamma\hat\iota_t^*\d(\frac{\partial \hat\iota_t}{\partial t}\lrcorner\Xi)+\int_\gamma\hat\iota_t^*(\frac{\partial \hat\iota_t}{\partial t}\lrcorner\d\Xi)\\
  &=\int_\gamma\hat\iota_t^*(\frac{\partial \hat\iota_t}{\partial t}\lrcorner q^*\bar\rho)=\int_\gamma\iota_t^*(\frac{\partial \iota_t}{\partial t}\lrcorner \bar\rho),
 \end{align*}
where we have used Stokes' theorem to cancel one term.

 The Lagrangian neighbourhood theorem allows us to identify $\bar\rho$ with the standard symplectic structure on $T^*L$, the immersion $\iota_t$ in $M$ with the immersion $(\mbox{Id},t\alpha)$ in $T^*L$ and $\frac{\partial \iota_t}{\partial t}$ with $\alpha$, viewed as a vertical vector field defined along the image of the immersion $(\mbox{Id},t\alpha)$.

The standard symplectic structure can be written locally as $dx^i\wedge dy^i$, so $\frac{\partial \iota_t}{\partial t}\lrcorner \bar\rho_{|\iota_t(x)}$ can be identified with $\alpha_k\partial y_k\lrcorner (dx^i\wedge dy^i)=-\alpha(x)$. Using $(\mbox{Id},t\alpha)$ this pulls back to $-\alpha=-\d f$.  (This calculation is essentially just the statement that the standard symplectic structure on $T^*L$ is $-\d\tau$ where $\tau$ is the tautological 1-form on $T^*L$.) Stokes' theorem allows us to conclude that $\frac{d}{dt}\int_\gamma \xi_J[\iota_t]=0$. For $t=0$, $\int_\gamma \xi_J[\iota]=0$ so we are done.
\end{proof}

It is convenient to introduce the notation 
$$C^{k,a}_0(L):=\{f\in C^{k,a}(L):\int_L f\vol_J[\iota]=0\}.$$
Set $\tU:=\{f\in C^{k+3,a}_0(L): \d f\in\mathcal{U}\}$. This parametrizes the exact $\bar\rho$-Lagrangian submanifolds in $M$ near $\iota$, allowing us to further restrict the Maslov map. 
With the hypothesis of Lemma \ref{l:exact} this map can be reduced to functions: given $f\in\tU$, let $\tilde{F}(f)$ denote the unique function with integral zero (where the integration is with respect to the volume form $\vol_J[\iota]$)
 whose differential is $\xi_J[\iota_{\d f}]$.  We may construct $\tilde{F}(f)$ explicitly by integration along paths in $L$: the definition is independent of the choice of path since $\xi_J[\iota_{\d f}]$ is exact.
   Since $\xi_J[\iota_{\d f}]\in C^{k,a}$,
by construction we then have that $\tilde{F}(f)\in C^{k+1,a}_0$.    
We thus obtain the \textit{scalar Maslov map}
\begin{equation}\label{F.eq_bis}
 \tilde{F}:\tU\to C^{k+1,a}_0(L),\ \ \d(\tilde{F}(f))=\xi_J[\iota_{\d f}].
\end{equation}

Now assume the initial immersion satisfies $\xi_J[\iota]=0$. We can then compute the linearisation $\mathcal{L}:=DF_{|0}$ by restricting the formula in Proposition \ref{lin.xi.prop} to the directions $JY$. Using (\ref{eq:weinstein}) we can identify $JY$ with a 1-form $\alpha$: specifically, one obtains the relationship
\begin{equation}\label{eq:identify}
 \alpha:=\bar\rho(JY,\cdot)=-\bar\rho(Y,J\cdot)=-\overline\Ric(Y,\cdot).
\end{equation}
Recall that any definite symmetric bilinear form on $TL$ provides \textit{musical isomorphisms} $\sharp$, $\flat$ between $T^*L$ and $TL$. We will write them in superscript when using the metric $g$, in subscript when using the restriction $\iota^*{\overline{\Ric}}$ of the ambient Ricci tensor. It follows from (\ref{eq:identify}) that $Y^\flat=-(\alpha_\sharp)^\flat$. 

Alternatively, let $A$ denote the $g$-self-adjoint operator on $T_pL$ defined by the identity $\iota^*\overline{\Ric}(\cdot,\cdot)=g(A\cdot,\cdot)$. Our assumptions on $M$ imply that $A$ is diagonalizable with 
non-zero eigenvalues, all with the same sign. 
 From the identities $\iota^*\overline{\Ric}(\alpha_\sharp,\cdot)=\alpha(\cdot)=g(\alpha^\sharp,\cdot)$ we deduce that $\alpha_\sharp=A^{-1}\alpha^\sharp$. It follows that $(\alpha_\sharp)^\flat=A^{-1*}\alpha$.

With these identifications we obtain
\begin{align}\label{L.eq}
\mathcal{L}(\alpha)=&-\d\big(\rho_J^{-1}\d^*(\rho_J(\alpha_\sharp)^\flat)\big)+\alpha\nonumber\\
=&-\d\big(\rho_J^{-1}\d^*(\rho_JA^{-1*}\alpha)\big)+\alpha.
\end{align}

Likewise, the linearisation $\tL:=D\tF_{|0}$ is 
\begin{align}\label{tL.eq}
\tL(f)=&-\rho_J^{-1}\d^*(\rho_J(\d f_\sharp)^\flat)+f\nonumber\\
=&-\rho_J^{-1}\d^*(\rho_JA^{-1*}\d f)+f,
\end{align}
since this satisfies $\d\tL(f)=\mathcal{L}(\d f)$ and $\int_L\tL(f)\vol_J[\iota]=0$.

When $M$ is K\"ahler--Einstein with $\overline{\Ric}=\lambda \bar g$ these formulae simplify because $\lambda(\alpha_\sharp)^\flat=\alpha$. For example, (\ref{tL.eq}) becomes $\tL(f)=-\lambda^{-1}\Delta_g f+f,$
  where $\Delta_g:=\d^*\d$ is the usual Hodge Laplacian.

\subsection{Analytic properties of the linearisation}\label{ss:lin_is_iso}
We can now prove the main results of this section.

\begin{prop}\label{prop:iso}
Let $M$ be a K\"ahler manifold with negative definite Ricci tensor and let $\iota:L\to M$ be a compact $J$-minimal immersion.

In terms of the Lagrangian neighbourhood theorem applied to $\iota$, the linearisation of the Maslov map $\xi_J$, restricted to $\bar\rho$-Lagrangian deformations as in (\ref{L.eq}), is a Banach space isomorphism 
 $\mathcal{L}:\mathcal{Z}^{k+2,a}\to\mathcal{Z}^{k,a}$.
\end{prop}
\begin{proof}  Recall that we let $A$ denote the $g$-self-adjoint operator on $T_pL$ defined by the identity $\iota^*\overline{\Ric}(\cdot,\cdot)=g(A\cdot,\cdot)$. We can extend $\mathcal{L}$ from closed 1-forms to the space of all 1-forms via the map
\begin{equation*}
 P:C^{k+2,a}(\Lambda^1(L))\to C^{k,a}(\Lambda^1(L)), \ \ P(\alpha):=\mathcal{L}(\alpha)-A^*\rho_J^{-1}\d^*(\rho_J A^{-2*}\d\alpha).
\end{equation*}
We first show that $P$ is elliptic. Up to lower order terms $\mathcal{L}(\alpha)$ coincides with $-A^{-1*}(\d\d^*\alpha)$. Given $\zeta\in T_p^*L$, we thus find that its principal symbol is 
$$\sigma(\mathcal{L})_p(\zeta):T_p^*L\to T_p^*L, \ \ \alpha\mapsto A^{-1*}\big(\zeta\wedge(\zeta^\sharp\lrcorner\alpha)\big).$$
Similarly we compute that the principal symbol of $P$ is
$$\sigma(P)_p(\zeta):T_p^*L\to T_p^*L, \ \ \alpha\mapsto A^{-1*}\big(\zeta\wedge(\zeta^\sharp\lrcorner\alpha)+\zeta^\sharp\lrcorner(\zeta\wedge\alpha)\big)=|\zeta|^2A^{-1*}\alpha.$$
It is thus a self-adjoint isomorphism, negative-definite by hypothesis on $M$, so $P$ is elliptic. Standard theory then implies that $P$ is Fredholm.

We now show that $P$ is injective.  For convenience, for $k$-forms $\alpha,\beta$ on $L$ we define
\begin{equation*}
 \langle \alpha,\beta\rangle_J=\int_L\rho_J\alpha\wedge *\beta=\int_Lg(\alpha,\beta)\vol_J\quad\text{and}\quad \|\alpha\|_{J}^2=\langle\alpha,\alpha\rangle_J.
\end{equation*}
 Notice that
\begin{align*}
\langle P(\alpha), A^{-1*}\alpha\rangle_J
=&-\langle \d(\rho_J^{-1}\d^*(\rho_J A^{-1*}\alpha)),A^{-1*}\alpha\rangle_J+\langle\alpha,A^{-1*}\alpha\rangle_J\\
&-\langle A^*\rho_J^{-1}\d^*(\rho_J A^{-2*}\d\alpha),A^{-1*}\alpha\rangle_J\\
=&-\|\rho_J^{-1}\d^*(\rho_JA^{-1*}\alpha)\|^2_J-\|\sqrt{-A^{-1}}^*\alpha\|^2_J-\|A^{-1*}\d\alpha\|^2_J,
\end{align*}
where we use $\vol_J=\rho_J\vol_g$ to integrate by parts and the fact that $-A^{-1}$ is positive definite, so it has a well-defined square root. Since the right-hand side is non-positive, we see that $P(\alpha)=0$ implies  each term on the right-hand side is zero; in particular, $\alpha=0$.

One can check that the formal adjoint of $P$ with respect to the $L^2$-metric $\langle.,.\rangle_J$ defined using $\vol_J$ is the map
\begin{equation*}
 P^*(\alpha):=-A^{-1*}\d(\rho_J^{-1}\d^*(\rho_J\alpha))+\alpha-\rho_J^{-1}\d^*(A^{-2*}\rho_J\d (A^*\alpha)).
\end{equation*}
The same reasoning shows that $P^*$ is injective, using the identity
\begin{align*}
\langle P^*(\alpha),A^*\alpha\rangle_J=&-\|\rho_J^{-1}\d^*(\rho_J\alpha)\|^2_J-\|\sqrt{-A}^*\alpha\|^2_J-\|A^{-1*}\d(A^*\alpha)\|^2_J.
\end{align*}
Thus, $P$ is a Banach space isomorphism. 

We must now argue that its restriction $\mathcal{L}$ is an isomorphism on the appropriate spaces. We already know that $\mathcal{L}(\mathcal{Z}^{k+2,a})\subseteq\mathcal{Z}^{k,a}$. To prove equality, assume $\beta \in\mathcal{Z}^{k,a}$. We know there exists some $\alpha\in C^{k+2,a}(\Lambda^1(L))$ such that $P(\alpha)=\beta$. Then, after simplifying, one finds
\begin{equation*}
 0=\d\beta=\d P(\alpha)=\d\alpha-\d (A^*\rho_J^{-1}\d^*(\rho_J A^{-2*}\d\alpha))
\end{equation*}
so $\d\alpha=\d (A^*\rho_J^{-1}\d^*(\rho_J A^{-2*}\d\alpha))$. Taking the $L^2$-inner product $\langle.,.\rangle_J$ of both sides of this equality with $A^{-2*}\d\alpha$, one finds
\begin{equation*}
\|A^{-1*}\d\alpha\|^2_J=-\|\rho_J^{-1}\sqrt{-A}^*\d^*(\rho_JA^{-2*}\d\alpha)\|^2_J.
\end{equation*}
As the left-hand side is non-negative, whilst the right-hand side is non-positive, we deduce both must vanish and hence $\d\alpha=0$. This proves that $\mathcal{L}$ is an isomorphism.
\end{proof}

\begin{remark}\label{rem:iso} If $M$ is KE with $\overline{\Ric}=\lambda \bar g$, $\lambda<0$, then the $J$-minimal immersion is minimal Lagrangian. In this case the calculations are much simpler because $\mathcal{L}=-\lambda^{-1}\d\d^*+\mbox{Id}$. It thus extends to $-\lambda^{-1}\Delta_g+\mbox{Id}$ on the space of all 1-forms. This operator is self-adjoint and is a Banach space isomorphism.  Hence, $\mathcal{L}$ is an isomorphism.
\end{remark}

\begin{prop}\label{prop:iso_bis}
Let $M$ be a K\"ahler manifold with definite Ricci tensor and $\iota:L\to M$ a compact $J$-minimal immersion.

In terms of the Lagrangian neighbourhood theorem applied to $\iota$, the linearisation of the Maslov map $\xi_J$, restricted to exact $\bar\rho$-Lagrangian deformations as in (\ref{tL.eq}), is an elliptic self-adjoint operator $\tL:C^{k+3,a}_0(L)\to C^{k+1,a}_0(L)$. 
When the Ricci tensor is negative definite, $\tL$ is a Banach space isomorphism.
\end{prop}

\begin{proof}
To prove that $\tL$ is self-adjoint with respect to the $L^2$-metric $\langle.,.\rangle_J$ defined using $\vol_J$, choose any function $h$ on $L$. The result then follows from the identity
\begin{align*}
\langle \rho_J^{-1}\d^*(\rho_J A^{-1*}\d f),h\rangle_J&=
\langle A^{-1^*}\d f, \d h\rangle_J=\langle\d f, A^{-1^*}\d h\rangle_J\\
&=\langle f,\rho_J^{-1}\d^*(\rho_J A^{-1*}\d h)\rangle_J.
\end{align*}
To prove ellipticity it suffices to study the highest order terms of $\tL(f)$, \textit{i.e.}~of $-\d^*(A^{-1*}\d f)$. Using $g$-normal coordinates at $p\in L$ which diagonalize $A(p)$, the principal symbol at $\zeta\in T^*_pL$ is (using summation convention)
\begin{equation*}
\sigma(\tL)_p(\zeta):=\lambda^{-1}_i\zeta_i^2,
\end{equation*} 
which vanishes if and only if $\zeta=0$, as desired. It follows that $\tL$ is elliptic and thus Fredholm.

To prove that $\tL$ is injective when the Ricci curvature is negative, assume $\tL(f)=0$. Then
\begin{equation*}
0=\langle\tL(f),f\rangle_J=\|\sqrt{-A^{-1}}^*\d f\|_J^2+\|f\|^2_J\geq 0,
\end{equation*}
with equality if and only if $f=0$.  It follows that $\tL$ is an isomorphism.
\end{proof}

\section{Persistence and uniqueness results}\label{s:Lag.KE}

In general, given a Lagrangian  $L$ in a symplectic manifold $(M,\omega)$ and a curve $\omega_t$ of symplectic structures on $M$ with $\omega_0=\omega$, 
there are obstructions to finding a corresponding curve of Lagrangian deformations of $L$. 

If however $\omega_t=\omega+\d\alpha_t$, \textit{i.e.}~they are all cohomologous, then a theorem of Moser shows that the $\omega_t$ are all symplectomorphic. Indeed, the non-degeneracy condition on $\omega_t$ allows us to find a curve a vector fields $X_t$ such that $\omega_t(X_t,\cdot)+\alpha_t=0$. Let $\phi_t\in \Diff(M)$ denote the flow of $X_t$. Then
\begin{equation}
 \frac{d}{dt}\phi_t^*\omega_t=\phi_t^*(\mathcal{L}_{X_t}\omega_t+\dot{\omega}_t)=\phi_t^*\d (X_t\lrcorner \omega_t+\alpha_t)=0,
\end{equation}
where we use the fact that $\d\omega_t=0$. It follows that $\phi_t^*\omega_t=\phi_0^*\omega_0=\omega$. As an application of this result, we find that any initial $\omega$-Lagrangian immersion $\iota:L\to M$ can be perturbed to the family of $\omega_t$-Lagrangian immersions $\iota_t:=\phi_t\circ\iota$.

We are interested in the following two cases.

\begin{prop}\label{prop:moser}
 Let $(M,J,\overline{g},\overline{\omega})$ be a K\"ahler manifold with definite Ricci tensor and let $\iota:L\to M$ be a compact $J$-minimal immersion. 
 \begin{itemize}
  \item Assume $M$ is KE so that $\iota$ is minimal Lagrangian. Let $(J_t,\overline{g}_t,\overline{\omega}_t)$ be a curve of KE structures on $M$ which coincides with the original structure when $t=0$. Then there exists a curve of $\overline{\omega}_t$-Lagrangian immersions $\iota_t:L\to M$ with $\iota_0=\iota$.
  \item Let $(J,\bar g_t,\bar\omega_t)$ be a curve of K\"ahler structures in the same cohomology class as $\bar\omega$, which coincides with the original structure when $t=0$. Assume the corresponding Ricci 2-forms $\bar\rho_t$ are definite. Then there exists a curve of $\bar\rho_t$-Lagrangian immersions $\iota_t:L\to M$ with $\iota_0=\iota$.
 \end{itemize}

\end{prop}
\begin{proof}
In the first case, since $c_1(M)$ is an integral cohomology class it does not change with $t$. 
This implies that the corresponding Ricci 2-forms $\overline{\rho}_t$ are cohomologous, so by the KE condition the same is true for $\overline{\omega}_t$: here the assumption that the Einstein constant is non-zero is crucial. We can thus apply Moser's theorem to $(M,\bar\omega_t)$, as above.

In the second case the initial $\iota$ is $\bar\rho$-Lagrangian. Since the Ricci forms belong to $c_1(M)$ they are all cohomologous, so we can apply Moser's theorem to $(M,\bar\rho_t)$.
\end{proof}

\paragraph{Minimal Lagrangian submanifolds.} Let $M$ be a negative KE manifold and let $\iota:L\to M$ be minimal Lagrangian. 

It follows from \cite{White} that $\iota$ is locally unique in the space of all minimal submanifolds in $M$ (even without the Lagrangian condition). We sketch a simple alternative proof in Remark \ref{rem:uniqueness}, below.

Concerning persistence, we shall now prove one of our main results: minimal Lagrangians persist under small KE perturbations of a given negative KE structure.

\begin{thm}\label{thm:persistence}
Let $(M,J,\overline{g},\overline{\omega})$ be a negative K\"ahler--Einstein manifold and let $\iota:L\to M$ be a compact minimal Lagrangian. For any small K\"ahler--Einstein deformation 
 $(J',\overline{g}',\overline{\omega}')$ of the ambient structure there exists a unique minimal Lagrangian $\iota':L\to (M,J',\overline{g}',\overline{\omega}')$ near $\iota$.
\end{thm}

\begin{proof}
Since $(J',\overline{g}',\overline{\omega}')$ is a deformation of the initial KE structure, there exists a curve of negative KE structures $(J_t,\overline{g}_t,\overline{\omega}_t)$ connecting them. Let $\iota_t$ be a curve of $\bar\omega_t$-Lagrangians as in Proposition \ref{prop:moser}. Consider the Lagrangian neighbourhood theorem for each $t$, applied to $\iota_t$: if the deformation is sufficiently small, we obtain a $t$-independent open neighbourhood $V$ of $L$ in $T^*L$ and, 
for each $t$, a tubular neighbourhood $T_t$ of $\iota_t(L)$ in $M$ and a symplectomorphism $\Phi_t:V\to (M,\overline{\omega}_t)$. 

As in \eqref{closed.forms.eq}, let $\mathcal{U}$ be the subset of closed 1-forms $\alpha$ in $\mathcal{Z}^{k+2,a}$ (defined using a fixed metric $g$ on $L$) whose graphs are contained in $V$. We can use these forms to parametrize the Lagrangian deformations $\iota_{\alpha,t}$ of $\iota_t$. We thus obtain restricted Maslov maps
\begin{equation}\label{Ft.eq}
F:\mathcal{U}\times [0, \epsilon)\to \mathcal{Z}^{k,a},\ \ F(\alpha,t)=\xi_{J_t}[\iota_{\alpha,t}]. 
\end{equation}
Theorem \ref{thm:minlag_is_Jmin} shows that $F(\alpha,t)=0$ if and only if $\iota_{\alpha,t}:L\to (M,\overline{\omega}_t)$ is 
a $C^{k+2,a}$ minimal Lagrangian. According to the standard regularity theory for minimal submanifolds, this coincides with the space of smooth minimal Lagrangians. 

Consider the linearisation of \eqref{Ft.eq} at $(0,0)$:
$$DF_{|(0,0)}:\mathcal{Z}^{k+2,a}\times\R\to\mathcal{Z}^{k,a}.$$
Restricting to $\mathcal{Z}^{k+2,a}\times\{0\}$ we obtain the map discussed in Proposition \ref{prop:iso}; for this purpose, the simpler result of Remark \ref{rem:iso} actually suffices. The full linearisation is thus surjective and its kernel is isomorphic, under projection, to $\R$. Applying the Implicit Function Theorem shows that, for any small $t$, there is a unique solution to the equation $F(\alpha,t)=0$.
\end{proof}

\begin{remark}\label{rem:white_bis}
It follows from \cite[Theorem 2.1]{White} that $L$ persists as a minimal submanifold under arbitrary small perturbations of the metric $\overline{g}$. However those methods do not show that, under KE perturbations, these deformed submanifolds are Lagrangian with respect to the new K\"ahler forms. The analogue of Theorem \ref{thm:persistence} in the special case $n=2$ was proved in \cite{YILee} but the techniques used there apply only to that dimension.
\end{remark}

\begin{remark}\label{rem:uniqueness} We prove here that, when $M$ is KE with negative scalar curvature, a minimal Lagrangian $\iota$ is locally unique in the space of all minimal submanifolds in $M$.

As $\iota$ is Lagrangian, by identifying $L$ with its image  we have identifications $T_pL^\perp\simeq T_pL\simeq T_p^*L$ for all $p\in L$. Using the tubular neighbourhood theorem we can thus identify submanifolds which are $C^1$-close to $\iota(L)$, as well as the corresponding mean curvature vector fields, with 1-forms on $L$. We then obtain a map 
\begin{equation*}
G:\mathcal{V}\subseteq C^{k+2,a}(T^*L)\to C^{k,a}(T^*L),\ \ G(\alpha):=H(\iota_\alpha),
\end{equation*}
where $\iota_\alpha$ is the immersion corresponding to $\alpha$ and $\mathcal{V}$ is some open neighbourhood of $0$.  By the second variation formula for the Riemannian volume functional at a minimal Lagrangian in a K\"ahler manifold (see \cite[Chapter V $\S$4]{Chen} and \cite[Theorem 3.5]{Oh}), the linearisation of $G$ at $0$ is
\begin{equation*}
\mathcal{L}_G(\alpha):=DG_{|0}(\alpha)=(-\Delta_g+\lambda)\alpha,
\end{equation*}
where $\overline{\Ric}=\lambda\bar{g}$.  Clearly $\mathcal{L}_G$ is elliptic and, since $\lambda<0$, it defines a Banach space isomorphism 
$\mathcal{L}_G:C^{k+2,a}(T^*L)\to C^{k,a}(T^*L)$.  In particular there 
 exists a constant $c_1>0$ such that, for all $\alpha\in C^{k+2,a}(T^*L)$,
\begin{equation}\label{eq:LG.est}
\|\alpha\|_{C^{k+2,a}}\leq c_1\|\mathcal{L}_G(\alpha)\|_{C^{k,a}}.
\end{equation} 

Since $G$ depends at most on second derivatives of $\alpha$, we may write
\begin{equation}\label{eq:Gdecomp}
G(\alpha)=\mathcal{L}_G(\alpha)+\mathcal{Q}_G(\alpha),
\end{equation}
where $\mathcal{Q}_G(\alpha)$ is a function of up to second derivatives of $\alpha$ whose value and first derivatives vanish at $\alpha=0$. It follows that there exists a constant $c_2>0$ such that
\begin{equation}\label{eq:QG.est}
\|\mathcal{Q}_G(\alpha)\|_{C^{k,a}}\leq c_2\|\alpha\|_{C^{k+2,a}}^2.
\end{equation}
Suppose now that we have a sequence of minimal submanifolds converging to $L$ and not equal to $L$.  We then obtain a sequence $\alpha_n\in\mathcal{V}\setminus\{0\}$, $\alpha_n\to 0$, satisfying $G(\alpha_n)=0$. Then by \eqref{eq:Gdecomp} we have that 
\begin{equation*}
\mathcal{L}_G(\alpha_n)=-\mathcal{Q}_G(\alpha_n).
\end{equation*}
Taking the norm of both sides and using the estimates \eqref{eq:LG.est} and \eqref{eq:QG.est} we see that
\begin{equation*}
c_1^{-1}\|\alpha_n\|_{C^{k+2,a}}\leq \|\mathcal{L}_G(\alpha_n)\|_{C^{k,a}}
=\|\mathcal{Q}_G(\alpha_n)\|_{C^{k,a}}\leq c_2\|\alpha_n\|_{C^{k+2,a}}^2,
\end{equation*}
obtaining a contradiction as $n\to\infty$.
\end{remark}

\paragraph{$J$-minimal submanifolds.}
Let $M$ be a K\"ahler manifold with negative Ricci curvature and let $\iota:L\rightarrow M$ be a compact $J$-minimal immersion. The second variation of the $J$-volume at $\iota$ is computed in \cite{LPtotally}, showing that $\iota$ is strictly stable so there are no Jacobi fields. It follows that $\iota$ has no $J$-minimal deformations. We can now prove the stronger result that $\iota$ is actually isolated.
\begin{thm}\label{thm:uniqueness}
Let $M$ be a K\"ahler manifold with negative Ricci curvature and let $\iota:L\to M$ be a compact $J$-minimal immersion. Then $\iota$ is locally unique, in the sense that there exists an open $C^{2,a}$-neighbourhood of $\iota$ containing no other $J$-minimal immersions.
\end{thm}
\begin{proof}
Using the Lagrangian neighbourhood theorem applied to $\iota$, any  $C^{2,a}$-smooth $J$-minimal immersion $\iota':L\to M$ close to $\iota$ would be parametrised by a closed $1$-form $\alpha'\in\mathcal{U}\subseteq \mathcal{Z}^{2,a}$, satisfying $F(\alpha')=0$. However $\mathcal{L}=DF_{|0}$ is a Banach space isomorphism so the Inverse Function Theorem, applied to the map $F$, shows that $\alpha=0$ is locally the unique zero of $F$
\end{proof}

We now turn to the persistence question. Here we must confront the fact that the $J$-minimal equation is not elliptic so it is unclear whether general $J$-minimal submanifolds are automatically smooth. As explained in the proof below, persistence within the $C^{k,a}$-category can be proved as in Theorem \ref{thm:persistence}. The following theorem shows that, using the exact Maslov map, we can also prove persistence within the $C^\infty$-smooth category.
\begin{thm}\label{thm:persistence_bis}
Let $(M,J,\overline{g},\overline{\omega})$ be a K\"ahler manifold with negative Ricci curvature and let $\iota:L\to M$ be a smooth compact $J$-minimal immersion. For any small K\"ahler deformation $(\overline{g}',\overline{\omega}')$ in the same cohomology class, there exists a unique smooth $J$-minimal immersion $\iota':L\to (M,J,\overline{g}',\overline{\omega}')$ near $\iota$.
\end{thm}

\begin{proof}
As in the proof of Theorem \ref{thm:persistence}, choose a curve of K\"ahler structures $\bar\omega_t$ connecting $\bar\omega$ and $\bar\omega'$ within the same cohomology class. If the deformation is sufficiently small, we may assume the corresponding Ricci 2-forms $\bar\rho_t$ are negative so we can use Proposition \ref{prop:moser} to build a curve of smooth $\bar\rho_t$-Lagrangian immersions $\iota_t:=\phi_t\circ\iota:L\to M$.

We could now proceed exactly as in the proof of Theorem \ref{thm:persistence}, obtaining a unique family of $J$-minimal totally real immersions. However, these immersions would only be $C^{k+2,a}$-smooth. 
To improve on this result we shall switch to the scalar Maslov map introduced in (\ref{F.eq_bis}).  As a first step this requires proving that the Maslov forms $\xi_J[\iota_t]$ are all exact, \textit{i.e.}~that for all closed curves $\gamma$ in $L$, $\int_\gamma\xi_J[\iota_t]=0$. We will proceed as in Lemma \ref{l:exact}, this time taking into account that $\Xi$ depends on $t$. Thus
 \begin{align*}
  \frac{d}{dt}\int_\gamma \xi_J[\iota_t]&=\int_\gamma\frac{d}{dt}\hat\iota_t^*\Xi_t=\int_\gamma\hat\iota_t^*(\mathcal{L}_{\frac{\partial \hat\iota_t}{\partial t}}\Xi_t+\dot{\Xi}_t)\\
  &=\int_\gamma\hat\iota_t^*\d(\frac{\partial \hat\iota_t}{\partial t}\lrcorner\Xi_t)+\int_\gamma\hat\iota_t^*(\frac{\partial \hat\iota_t}{\partial t}\lrcorner\d\Xi_t+\dot{\Xi}_t)\\
  &=\int_\gamma\iota_t^*(\frac{\partial \iota_t}{\partial t}\lrcorner \bar\rho_t)+\int_\gamma\hat\iota^*\dot{\Xi}_t,
 \end{align*}
where we used Stokes' theorem to cancel one term. 

Since all the $\bar\rho_t$ are cohomologous we can write $\bar\rho_t=\bar\rho+\d\d^cf_t$, for some curve of functions $f_t$ on $M$. It follows that $\frac{\partial}{\partial t}\bar\rho_t=\d\d^c\dot{f}_t$. Recall from the proof of Moser's theorem that the vector field $X_t$ whose flow is $\phi_t$ satisfies the equation $\bar\rho_t(X_t,\cdot)+\d^c\dot f_t=0$. Since $\d\Xi_t=q^*\bar\rho_t$, we find that $\dot\Xi_t=q^*(\d^c\dot f_t+\d h_t)$, for some curve of functions $h_t$ on $M$. Again using Stokes' theorem, we conclude that 
\begin{equation}
\frac{d}{dt}\int_\gamma \xi_J[\iota_t]=\int_\gamma\iota_t^*(X_t\lrcorner\bar\rho_t+\d^c\dot f_t)=0.
\end{equation}
Since initially $\xi_J[\iota]=0$, we conclude that all $\xi_J[\iota_t]$ are exact.

As in Theorem \ref{thm:persistence} we can use the Lagrangian neighbourhood theorem applied to $\iota_t$ to set up a family of scalar Maslov maps as in (\ref{F.eq_bis}),
\begin{equation}
 \tilde{F}:\tU\times [0,\epsilon)\to C^{k+1,a}_0(L).
\end{equation}
The zero set of these maps parametrize the exact $J$-minimal totally real perturbations of $\iota_t$. As in the proof of Theorem \ref{thm:persistence}, using Proposition \ref{prop:iso_bis} and the Implicit Function Theorem we obtain a $J$-minimal totally real immersion for each $t$. Since $\tilde{F}(f,t)=0$ is a scalar elliptic equation, these immersions are smooth. This argument also shows that they are unique within the category of exact deformations. Theorem \ref{thm:uniqueness} yields the stronger result that they are the unique $J$-minimal immersions near $\iota$.
\end{proof}

\section{Examples}\label{s:examples}

One of the main sources of examples of minimal Lagrangian submanifolds is the following.

Let $p$ be a homogeneous polynomial of degree $d>n+1$ on $\C^{n+1}$ such that 
\[
M=\{[z]\in\CP^n\,:\, p(z)=0\}
\]
is  smooth. Let $J$ be the induced complex structure. The adjunction formula shows that the class $c_1(M,J)$ is negative, so the general existence theory of KE metrics proves that $(M,J)$ 
has a unique KE $2$-form $\overline{\omega}$ in the class $-c_1(M,J)$. The corresponding metric $\overline{g}$ has negative scalar curvature. 

If the coefficients of the polynomial $p$ are real, $M$ is invariant under complex conjugation on $\CP^n$: this defines by restriction a \textit{real structure} $\tau\in \Diff(M)$, \textit{i.e.}~an anti-holomorphic involution: 
$\tau^*J=-J$. The $2$-form $\tau^*\overline{\omega}$ is then KE for $(M,\tau^*J)$.
 Clearly $-\overline{\omega}$ has the same property, so by uniqueness $\tau^*\overline{\omega}=-\overline{\omega}$. It follows that $\tau$ is an isometry of the metric $\overline{g}$.

Using the above facts, one can check that the fixed point set $L$ of $\tau$ (if non-empty), often called the \textit{real locus}, is a totally geodesic (thus minimal) Lagrangian submanifold of the KE manifold 
$(M,\overline{g},J,\overline{\omega})$. 

It is clear that this method requires the existence of a special symmetry of $(M,J)$, namely a real structure. It is a good question what happens when such a symmetry does not exist.

In the above setting the corresponding moduli space of KE manifolds can locally be constructing simply by perturbing the polynomial $p$. Using Theorem \ref{thm:persistence} we can now show that there exists a minimal Lagrangian submanifold even if $M$ is deformed so as to lose its real structure.

Concerning $J$-minimal submanifolds, up to now the only known examples are possibly Borrelli's \cite{Borrelli} non-compact examples in $\C^n$, but it is clear that the Ricci-flat case is rather special. Using Theorem \ref{thm:persistence_bis} we can now prove the existence of compact examples in the negative Ricci case.

\begin{cor}
Let $p$ be a homogeneous polynomial on $\C^{n+1}$, of degree $d>n+1$ and with real coefficients, such that 
\[
M=\{[z]\in\CP^n\,:\, p(z)=0\},
\]
endowed with the induced complex structure $J$ and negative KE structure has non-empty real locus $L$.  

\begin{itemize}
\item Let $(M', J')\subseteq \CP^n$ be defined by a polynomial $p'$ obtained by perturbing $p$. 
We endow $M'$ with the corresponding negative K\"ahler--Einstein metric $(\bar g',\bar \omega')$.    

There exists $\epsilon>0$ such that, if the imaginary parts of the coefficients of $p'$ are smaller than $\epsilon$, then there exists a locally unique compact minimal Lagrangian submanifold $L'$ in $M'$. 
\item Let $(M',J,\bar g',\bar\omega')$ denote any small perturbation of $M$ in the same K\"ahler class, with negative Ricci curvature. Then there exists a locally unique compact $J$-minimal submanifold $L'$ in $M'$.
\end{itemize}
These submanifolds are isotopic to the real locus of $M$.
\end{cor}

A second way to obtain examples of minimal Lagrangians is, somewhat trivially, via products. The deformation theory of complex manifolds shows that the only deformations of products of negative KE manifolds are those obtained by deforming each factor. Theorem \ref{thm:persistence} shows that any minimal Lagrangians will persist. We thus obtain the following result.
\begin{cor}
Let $\iota_1:L_1\to M_1$ and $\iota_2:L_2\to M_2$ 
be compact minimal Lagrangians in negative K\"ahler--Einstein manifolds with the same scalar curvature. 
Let $M_1'$, $M_2'$ be KE deformations of $M_1$, $M_2$ with the same scalar curvature. 
Then the only minimal Lagrangian submanifold in $M_1'\times M_2'$ near $\iota_1\times\iota_2$ is of the form $\iota_1'\times\iota_2'$, 
where $\iota_i'$ is the minimal Lagrangian in $M_i'$ obtained by deforming $\iota_i$. 
\end{cor}

\end{document}